\documentclass[12pt,oneside,reqno]{amsart}

\usepackage[a4paper,left=25mm,top=33mm,right=25mm,bottom=33mm]{geometry}

\usepackage{amssymb,amsfonts}
\usepackage[all,arc]{xy}
\usepackage{enumerate}
\usepackage{siunitx}
\usepackage{enumitem}
\usepackage{mathtools}
\usepackage{mathrsfs}
\usepackage{dsfont}
\usepackage{amsfonts}
\usepackage{amssymb}
\usepackage{graphicx}
\usepackage{ragged2e}
\usepackage{amsthm}
\usepackage{amsmath}
\usepackage[mathscr]{euscript}
 \let\mathscr\relax
\usepackage[scr]{rsfso}
\usepackage{graphicx}
\usepackage{color}
\usepackage{float}
\usepackage{caption}
\usepackage[pdftex,bookmarksnumbered,bookmarksopen]{hyperref}
\hypersetup{hidelinks = true}

\DeclareMathOperator{\vol}{Vol}

\newtheorem{thm}{Theorem}[section]
\newtheorem{cor}[thm]{Corollary}
\newtheorem{prop}[thm]{Proposition}
\newtheorem{lem}[thm]{Lemma}

\newtheorem{claim}[thm]{Claim}

\theoremstyle{definition}

\theoremstyle{remark}
\newtheorem{rem}[thm]{Remark}

\numberwithin{equation}{section}

\begin{document}
		\title{ Planck-scale number of nodal domains for toral eigenfunctions}
	\begin{abstract}
We study the number of nodal domains in balls shrinking slightly above the Planck scale for \textquotedblleft generic'' toral eigenfunctions. We prove that, up to the natural scaling, the nodal domains count obeys the same asymptotic law as the global number of nodal domains. The proof, on one hand, uses new arithmetic information  to refine Bourgain's de-randomisation technique at Planck scale. And on the other hand, it requires a Planck scale version of Yau's conjecture which we believe to be of independent interest. 
	\end{abstract}
	\author{Andrea Sartori}
	\address{ Departement of Mathematics, King's College London, Strand, London WC2R 2LS, England, Uk }
	\email{andrea.sartori.16@ucl.ac.uk}
	\maketitle
	
	\section{Introduction}
\subsection{Laplacian eigenfunctions and nodal domains}	
	 Given a compact Riemannian surface $(M,g)$ without boundary, let $\Delta_g$ be the Laplace-Beltrami operator on $M$. There exists an orthonormal basis for $L^2(M,d\vol)$ consisting of eigenfunctions $\{f_{E_i}\}$	
\begin{align}
\Delta_g f_{E_i}+ E_i f_{E_i}=0  \nonumber
\end{align}
with $0=E_1<E_2\leq...$ listed with multiplicity, and $E_i\rightarrow \infty$.  The \textit{nodal set} of an eigenfunction $f_E$ is the zero set $Z(f_E):= \{x\in M: f_E(x)=0\}$ and it is the union of  smooth curves outside a finite set of points \cite{C}.   The  connected components of  $M\backslash Z(F_E)$ are called \emph{nodal domains} and we denote their number by $\mathcal{N}(f_E)$. The main object of our interest is to count the number of nodal domains of $f_E$. 

The celebrated Courant Nodal Domains Theorem \cite{CH} implies that there exists an explicit constant $C>0$ such that	
\begin{align}
\mathcal{N}(f_E) \leq C \cdot E. \label{Courant}
\end{align}
Stern \cite{ST} showed that, on some planar domains, there exists a sequence of eigenfunctions such that the eigenvalue grows to infinity, but $\mathcal{N}(f_E)= 2$, see also \cite{LEW} for a similar result on the two dimensional sphere. Jung and Zelditch \cite{JZ} proved that for most eigenfunctions on certain negatively curved manifolds $\mathcal{N}(\cdot)$ tends to infinity with the eigenvalue. Ingremeau \cite{I} also gave examples of eigenfunctions with $\mathcal{N}(\cdot) \rightarrow \infty$ on unbounded negatively-curved manifolds. 
\subsection{The Random Wave Model}
  For \textquotedblleft generic" eigenfunctions, the Random Wave Model proposed by Berry \cite{B1,B2} together with the breakthrough work of Nazarov and Sodin \cite{NS} assert that there exists a constant $c>0$ such that 
 \begin{align}
 \mathcal{N}(f_E)= c\cdot E (1+o(1)). \label{Bourgain}
 \end{align}
 Remarkably, Bourgain \cite{BU}  proved that there exist sequences of eigenfunction on the standard flat torus $\mathbb{T}^2=\mathbb{R}^2/\mathbb{Z}^2$ such that \eqref{Bourgain} holds. Subsequently, Buckley and Wigman \cite{BW} extended Bourgain's work to \textquotedblleft generic'' toral eigenfunctions. 
 
 We study a finer form of \eqref{Bourgain}: let $s>0$ and let $\mathcal{N}_{f_E}(s,z)$ be the number of nodal domains lying entirely inside the geodesic ball of radius $s$ around the point $z\in M$; then the Random Wave Model would also predict that
 \begin{align}
 \mathcal{N}_{f_E}(s,z)= c \cdot E (\pi s^2) (1+o(1))\label{Shrinking}
 \end{align}
 uniformly in $z$, provided that $s \cdot E^{1/2} \rightarrow \infty$, i.e. provided that the radius of the ball shrinks slightly above the \textit{Planck-scale}. We prove that \eqref{Shrinking} holds for \textquotedblleft generic \textquotedblright toral eigenfunctions with $s> E^{-1/2+ o(1)}$. 
\subsection{Statement of main results}
\label{main results}
Every Laplace eigenfunction on $\mathbb{T}^2$ can be written as 
\begin{align}
f(x)=f_E(x)=\sum_{ \substack{ \xi \in \mathbb{Z}^2\\|\xi|^2=E} }a_{\xi}e(\langle x,\xi \rangle)  \label{function}
\end{align}
where $\{a_{\xi}\}_{\xi}$ are complex coefficients and $e(\cdot)= e(2\pi i \cdot )$ (This normalisation implies that the eigenvalue is $4\pi E$, but we will make no distinction between $E$ and $4\pi E$). The eigenvalues are integers $E\in S:=\{E \in \mathbb{N}:  E= a^2+b^2 , \text{ for some} \hspace{1mm} a,b \in \mathbb{Z}\}$  and their multiplicity, which we denote by $N=N(E)$, is given by the number of lattice points on the circle of radius $\sqrt{E}$. Moreover, we assume that $\bar{a_{\xi}}=a_{-\xi}$, that is $f$ is real-valued, and  that $f$ is normalised via 
\begin{align}
||f||^2_{L^2(\mathbb{T}^2)}=\sum_{ |\xi|^2=E}|a_{\xi}|^2=1 \label{normalisation} . 
\end{align}

Thanks to (\ref{normalisation}), we can regard the set $(a_{\xi})_{\xi}$ as points on an $N$-dimensional complex sphere. Then, L\'{e}vy concentration of measure \cite[Theorem 2.3]{LE} implies that, most $a_{\xi}$ are small,  $|a_{\xi}|^2\leq (\log N)^{O(1)}/ N$ say,  with probability asymptotic to $1$. Therefore, we say that $f$ is \textit{flat} if, for all $\rho>0$
\begin{align}
 &\max_{|\xi|^2=E}|a_{\xi}|^2= o(N^{-1+\rho}) & \text{as} \hspace{3mm}N\rightarrow \infty. \nonumber
\end{align}   
Also, via \eqref{normalisation}, we associate to $f$ the probability measure on the unit circle $\mathbb{S}^1= \mathbb{R}/\mathbb{Z}$ 
\begin{align}
\mu_{f}=\sum_{ |\xi|^2=E} |a_{\xi}|^2\delta_{\xi/\sqrt{E}} \label{spectral measure}
\end{align} 
where $\delta_{\xi/\sqrt{E}}$ is the Dirac distribution at the point $\xi/\sqrt{E}$. Finally, we denote by $c_{NS}(\mu_{f})$ the \textit{Nazarov-Sodin} constant relative to the measure $\mu_{f}$. In order to present our main result, we differ the discussion about $c_{NS}(\cdot)$  to Section   \ref{Gaussian random fields} below. Our principal result is the following: 
\begin{thm}
	\label{theorem 3}
	There exists a density one subset \footnote{
		By a density one subset we mean a set $S'\subset S$ such that $\lim\limits_{X\rightarrow \infty } \frac{\#\{E\leq X: E\in S'\}}{\#\{E\leq X: E\in S\}}=1$. } $S'\subset S$  such that for all $\epsilon>0$ we have
	\begin{align}
	\mathcal{N}_f(s,z)= c_{NS}(\mu_{f})\pi s^2 E(1+o_{E\rightarrow \infty}(1))  \nonumber
	\end{align}
	uniformly for $f$ flat, $s> E^{-1/2+ \epsilon}$ and $z\in \mathbb{T}^2$. 
\end{thm}

\begin{rem}
	Using the main result in \cite{SA}, Theorem \ref{theorem 3} still holds if we take $s$ such that for all $m>0$ we have $s \cdot E^{1/2}/ (\log E)^m \rightarrow \infty$. For the sake of elegance of the presentation, we decided not to include it in this manuscript.  
\end{rem}

The sequence $\{\mu_{f}\}$, even in the special case $a_{\xi}=1/\sqrt{N}$ for all $\xi's$,  does not have a unique limit point with respect to the weak$^{\star}$ topology on $\mathbb{S}^1$ \cite{CI,KW, SA2}. Thus, in order to obtain an asymptotic behaviour for $N(f_E)$, we have to pass to subsequences. Kurlberg and Wigman proved \cite[Theorem 1.3]{KW2}  that if $\mu_{f}$ weak$^{\star}$ converges to some probability measure $\mu$ on $\mathbb{S}^1$, then  $c_{NS}(\mu_{f})= c_{NS}(\mu)(1+o(1))$. This implies the following version of Theorem \ref{theorem 3}: 
\begin{cor}
	Under the assumptions of Theorem \ref{theorem 3}, suppose that $\mu_{f}$ weak$^{\star}$ converges to some probability measure $\mu$ on $\mathbb{S}^1$, then 
	\begin{align}	
	\mathcal{N}_f(s,z)= c_{NS}(\mu)\pi s^2 E(1+o(1)). \nonumber
	\end{align}
	uniformly for $f$ flat, $s> E^{-1/2+ \epsilon}$ and $z\in \mathbb{T}^2$. 
\end{cor} 
\subsection{Nodal length in shrinking balls}
One of the novel aspects in the proof of Theorem \ref{theorem 3} is the study of the \textit{nodal length}, that is the Hausdorff measure of the nodal set, of toral eigenfunctions in shrinking balls. The main (open) question about the nodal length of Laplace eigenfunctions is the following conjecture of Yau: let $f_E$ be a Laplace eigenfunction with eigenvalue $E$ on a smooth, compact manifold without boundaries $M$, then 
\begin{align}
\sqrt{E}\ll_M \mathcal{L}(f_E)= \mathcal{H}^{n-1}\{x\in M: f(x)=0\}\ll_M \sqrt{E} \nonumber
\end{align} 	
Donnelly and Fefferman \cite{DF} showed that Yau's conjecture holds for  real-analytic manifolds. Recently, Logunov and Malinnikova \cite{LM,L1,L2} proved the lower-bound in the smooth case and gave a polynomial upper-bound.

As for the nodal domains count, the Random Waves Model suggests that, for \textquotedblleft generic" Laplace eigenfunctions,  a rescaled version of Yau's conjecture should hold at small scales, that is for any $z\in M$ 
\begin{align}
s\sqrt{E}\ll_M s^{-1}\mathcal{L}_f(s,z):= s^{-1}\mathcal{H}^{n-1}\{x\in B(s,z): f(x)=0\}\ll_M s\sqrt{E}. \nonumber
\end{align}
provided that $s$ shrinks slightly above Plank-scale. We prove the following:
\begin{prop}
	\label{claim}
	Let $f$ be as in \eqref{function} and let $\epsilon>0$, then 
	\begin{align}
	s\sqrt{E}\ll	\mathcal{L}_f(s,z)s^{-1}\ll s\sqrt{E} \nonumber
	\end{align}
	uniformly for $s> E^{-1/2+\epsilon}$ and $z\in \mathbb{T}^2$. 
\end{prop}
One particular aspect of Proposition \ref{claim} is that it holds for \textit{every} toral eigenfunction. This might fail on other surfaces:  on the  $2$-sphere $\mathbb{S}^2= \{x\in \mathbb{R}^3: ||x||^2=1\}$ one can consider the \textquotedblleft sectoral" harmonic $g(\theta,\phi)=\sin (m\phi) P^m_m(\cos(\theta))$ in spherical-coordinates, where $P^m_m(\cdot)$ is the associated Legendre polynomial. Then $\Delta g= -m(m+1) g$ and the upper-bound in Proposition \ref{claim} fails around the North Pole. 

\vspace{2mm}

\paragraph{\textit{Application to Laplace eigenfunctions on the square}} The proof of Proposition \ref{claim} is general enough that it can also address Laplace eigenfunctions on the square $[0,1]^2$ with either Dirichlet or Neumann boundary conditions. The study of  the nodal length of \textit{random} Laplace eigenfunctions on the square,  known as \textit{boundary adapted Arithmetic Random Waves}, was initiated by Cammarota, Klurman and Wigman \cite{CKW}. A major step in their work is to bound the expectation of the nodal length in squares of side $O(1)/\sqrt{E}$, where $E$ is the eigenvalue. We prove the following: 
\begin{prop}
	\label{claim2}
	Let $\tilde{f}$ be a Laplace eigenfunction on the square $[0,1]^2$ with either Dirichlet or Neumann boundary conditions and let $E$ be its eigenvalue. Then we have 
	\begin{align}
	\mathcal{L}_{\tilde{f}}(s,z)s^{-1}\ll s\sqrt{E} + N \nonumber
	\end{align}
	uniformly for $s>0$ and $z\in \mathbb{T}^2$, where $N=N(E)$ is  as in Section \ref{main results}.
\end{prop}
From Proposition \ref{claim2}, it follows that for any fixed $C>0$ we have 
\begin{align}
\mathcal{H}^1\{x\in B(C/\sqrt{E},z): \tilde{f}(x)=0 \}\ll \frac{N}{\sqrt{E}}. \label{1.1}
\end{align}
For \textit{random} Laplace eigenfunctions on the square, Cammarota, Klurman and Wigman \cite[Proposition 2.5]{CKW} showed that the bound	$\mathcal{H}^1\{x\in B(C/\sqrt{E},z): \tilde{f}(x)=0 \}\ll N^2/\sqrt{E}$ holds with high probability. So \eqref{1.1} not only refines \cite[Proposition 2.5]{CKW} but it also provides a deterministic results which does not rely on moments estimates. 
\subsection{Bourgain's de-randomisation in shrinking sets}
Another novel aspect in the proof of Theorem \ref{theorem 3} is an extension of Bourgain's de-randomisation technique to shrinking sets. Let $f$ be as in \eqref{function} and suppose that $a_{\xi}=1$ for all $\xi's$, moreover let $F_x(y)= f(x+y/\sqrt{E})$ for $y\in [-1/2,1/2]^2$.  Bourgain \cite{BU} showed that the assemble $\{F_x\}$, where $x$ is drawn uniformly at random from $\mathbb{T}^2$, approximates the Gaussian field with spectral measure the Lebesgue measure on $\mathbb{S}^1$ (see Section \ref{Gaussian random fields} below for some background on Gaussian fields). We use some new informations about sum of lattice points called \textit{quasi-spectral correlations} to show that this approximations still holds even when $x$ is drawn uniformly at random from $B(s,z)$ for $s>E^{-1/2+o(1)}$ and $z\in \mathbb{T}^2$, Proposition \ref{main prop} below. 

Since the proof of Proposition \ref{main prop} is quite technical, to give the reader an idea of how such properties of lattice points are exploited, we show here that $F_x(0)=f(x)$ approximates  a standard Gaussian random variable when $x$ is drawn uniformly at random from $B(s,z)$. Via the method of moments, we have to evaluate for $l \in \mathbb{N}$ 
\begin{align}
\frac{1}{\pi s^2}	\int_{B(s,z)} |F_x(0)|^{2l} dx&=\frac{1}{\pi s^2} \sum_{\xi_1,...,\xi_{2l}} \int_{B(s,z)}e(\langle x, \xi_1-\xi_2+...+\xi_{2l-1}-\xi_{2l}\rangle) dx .\nonumber
\end{align}
Separating the terms with $\xi_1 +...-\xi_{2l}=0$, known as\textquotedblleft$2l$-spectral correlations", from the other terms,  \textquotedblleft $2l$ spectral quasi-correlations" , we obtain 
\begin{align}
\frac{1}{\pi s^2}	\int_{B(s,z)} |f(x)|^{2l} dx= \sum_{\xi_1- \xi_2+...-\xi_{2l}=0}1 
+O\left( \sum_{|\xi_1- \xi_2+...-\xi_{2l}|>0} \frac{J_1(s|\xi_1- \xi_2+...-\xi_{2l}|)}{s|\xi_1- \xi_2+...-\xi_{2l}|} \right) \label{intro3}
\end{align} 
where $J_1(\cdot)$ is the Bessel function of the first kind. 

The main contribution to the first term in \eqref{intro3} comes from the diagonal solutions $\xi_1=\xi_2$,..., $\xi_{2l-1}=\xi_{2l}$ and their permutations, these contribute $2l!/(2^l\cdot l!)$.  Bombieri and Bourgain \cite{BB} showed that, for \textit{generic} $E\in S$, the \textquotedblleft off-diagonal" solutions have lower order as $N\rightarrow \infty$. Thus, the first term on the right hand side of \eqref{intro3} is asymptotic to $2l!/(2^l\cdot l!)$. We are left to show that the second term on the right hand side of \eqref{intro3} tends to $0$ as $N\rightarrow \infty$. Theorem \ref{semi} below implies that for \textit{generic} $E\in S$,   $s|\xi_1- \xi_2+...-\xi_{2l}|\geq E^{o(1)}$. Since Bessel functions decay at infinity, this implies that the second term in \eqref{intro3} tends to $0$, as required.
\subsection{Related results}The main body of results regarding statistics of Laplace eigenfunctions in shrinking sets concern their mass distribution. Let $f_E$ be a Laplace eigenfunction on a surface $M$,  then one is interested in finding the smallest $s$ such that $\int_{B(s,z)} |f|^2 d\vol= \pi s^2(1+o_{E\rightarrow \infty}(1))$. The celebrated Quantum Ergodicity Theorem \cite{DV,S,Z} asserts that, if the geodesic flow on M is ergodic, then one can take any \textit{fixed} $s>0$ for a density one subsequence of eigenfunctions.   Luo and Sarnak  \cite{LS} showed that, on the modular surface, one can take $s>E^{-\alpha}$ for some $\alpha>0$ for a density one subsequence, see also \cite{Y}. Hezari, Rivi\`{e}re \cite{HR}  and independently Han \cite{H} proved that, if $M$ has negative sectional curvature, then one can take $s> \log(E)^{-\alpha}$ for some small $\alpha>0$ for a density one subsequence. On $\mathbb{T}^2$ Lester and Rudnick \cite{LR} showed that $s> E^{-1/2+ o(1)}$, again for a density one subsequence.  

Granville and Wigman \cite{GW} and subsequently  Wigman and Yesha \cite{WY} studied the mass distribution of eigenfunctions on $\mathbb{T}^2$ \emph{at Planck scale} by drawing the centre of the ball randomly uniformly. They showed that, for certain eigenfunctions  the mass equidistributes in almost every ball, see also \cite{HU,HU1} for similar work on the modular surface. The author \cite{SA1} classified all limiting mass-distributions \emph{at Planck scale}  for \textquotedblleft generic" toral eigenfunctions. 

Results regarding the zero set are more modest: Benatar,  Marinucci and Wigman \cite{BMW} studied the behaviour of nodal length for \emph{random} toral eigenfunctions at scales $s=E^{-1/2+o(1)}$ and found the asymptotic law of the variance. 
To the best of the author's knowledge, our own Theorem \ref{theorem 3} is the only asymptotic result on nodal domains at small scales.  

\subsection{Notation}
Let $t \rightarrow \infty$ be some parameter, we say that the quantity $X=X(t)$ and $Y=Y(t)$   satisfy $X\ll Y$ , $X\gg Y$ if there exists some constant $C$, independent of $t$, such that $X\leq C Y$ and $X\geq CY$ respectively. We also write $O(X)$ for some quantity bounded in absolute value by a constant times $X$ and $X=o(Y)$ if $X/Y\rightarrow 0$ as $t\rightarrow \infty$, in particular we denote by $o(1)$ any function that tends to $0$ (arbitrarily slowly) as $x\rightarrow \infty$. We denote by $B(s,z)$ the (open) ball of radius $s$ with centre $z$, by $B(s)$ for the ball centred at $0$ and by $\overline{B}(s)$ the closure of $B(s)$. When the specific radius is unimportant, we simply write the ball as $B$ and $\frac{1}{2}B$ for the concentric ball with half the radius. Finally, we denote by $\Omega$ an abstract probability space where every random object is defined.  

\section{Preliminaries}
\subsection{Number theoretic background}
\label{NTpre}
Recall that $S=\{n\in \mathbb{N}: n= a^2+b^2 , \text{ for some} \hspace{1mm} \\ a,b \in \mathbb{Z}\}$. In this section we collect some number theoretic results that will be used to define the set $S'\subset S$ in Theorem \ref{theorem 3}. Let $E\in S$ and write its prime factorisation as $E= \prod_{p\equiv 1 \pmod 4}p ^{\alpha_p}\prod_{q\equiv 3 \pmod 4} q^{2\beta_q}$ where $\alpha_p,\beta_q\in \mathbb{N}$. It follows that $N(E)=4 \prod_{p\equiv 1 \pmod 4}(\alpha_p+1)$. Thus, by the divisor bound, we have
\begin{align}
N(E) \ll \exp \left( \frac{\log E}{\log\log E}\right). \label{divisor bound}
\end{align}
 Moreover, by the Erd\"{o}s-Kac Theorem \cite[Theorem 12.3]{E}, for almost all integers (representable as sum of two squares) the number $\#\{ p|E: p\equiv 1 \pmod 4\}\rightarrow \infty$ as $E\rightarrow \infty$.  So we also have the following lemma:  
\begin{lem}
	\label{N infinity}
	For a density one subset of $E\in S$, $N(E) \rightarrow \infty$ as $E\rightarrow \infty$.  
\end{lem}
 To state the next results we need some notation: let  $l\in \mathbb{N}$ and $E\in S$,  denote by $\mathcal{S}(l,E)$ the number of solutions to 
\begin{align}
\xi_1+...+ \xi_l=0 \label{14}
\end{align}
where $\xi_j\in \mathbb{Z}^2$ and $|\xi_j|^2=E$, that is \textit{$l$-spectral correlations}. When $l$ is odd, by congruence obstruction modulo $2$, there are no solutions to \eqref{14}. When $l$ is even, we have the following \cite[Theorem 17]{BB} and \cite[Lemma 4]{BU}:  

\begin{thm}[Bombieri-Bourgain]
	\label{BB}
Let $B=B(E)$ be an arbitrarily slow growing function of $E$, $l\in \mathbb{N}$ and $0<\gamma<1$. Then, for a density one subset of integers $E\in S$, we have 
\begin{align}
	\mathcal{S}(2l,E)= \frac{(2l)!}{2^l \cdot l!}N^{l} +O(N^{\gamma l}) \nonumber
\end{align}
uniformly for all $l\leq B$, where the constant implied in the notation is absolute. 
\end{thm}
Moreover, provided that is not zero, one can give a quantitative lower bound to the sum in \eqref{14}, see \cite[Theorem 1.4]{BMW} and the refinement \cite[Theorem 1.1]{SA}:

\begin{thm}
	\label{semi}
Let $B=B(E)$ be an arbitrarily slow growing function of $E$, $l\in \mathbb{N}$ and $Q=Q(E)$ be a function such that $Q \cdot E^{1/2}/(\log E)^m\rightarrow \infty$ for all $m\geq 0$. Then, for a density one subset of integers $E\in S$, we have 
\begin{align}
	 ||\xi_1+...+\xi_l||>Q\nonumber
\end{align}
uniformly for all choices of $\xi_1,...,\xi_l$ and $l\leq B$. 
\end{thm}
\textbf{The set $S'$}. We are now ready to define the subset in Theorem \ref{theorem 3}: let $S'$ be the set of $E\in S$ which satisfy the conclusion of Lemma \ref{N infinity},  Theorem \ref{BB} and Theorem \ref{semi}. By the discussion in this section, $S'$ has density one. 
\subsection{Gaussian fields background}
\label{Gaussian random fields}
We  briefly collect some definitions about Gaussian fields (on $\mathbb{R}^2$). A (real-valued) Gaussian field $F$ is a measurable map $F: \mathbb{R}^2 \times \Omega\rightarrow \mathbb{R}$ for some probability space  $\Omega$,  such that all  finite dimensional distributions $(F(x_1, \cdot),...F(x_n,\cdot))$ are multivariate Gaussian. $F$ is \textit{centred} if $\mathbb{E}[F]=0$ and \textit{stationary} if its law is invariant under translations $x\rightarrow x+\tau$ for $\tau \in \mathbb{R}^2$. The \textit{covariance} function of $F$ is 
\begin{align}
\mathbb{E}[F(x)\cdot F(y)]= \mathbb{E}[F(x-y)\cdot F(0)]. \nonumber
\end{align}
Since the covariance is positive definite, by Bochner's theorem, it is the Fourier transform of some measure $\mu$ on the $\mathbb{R}^2$. So we have 
\begin{align}
\mathbb{E}[F(x)F(y)]= \int_{\mathbb{R}^2} e\left(\langle x-y, \lambda \rangle\right)d\mu(\lambda). \nonumber
\end{align}
The measure $\mu$ is called the \textit{spectral measure} of $F$ and, since $F$ is real-valued, satisfies $\mu(-I)=\mu(I)$ for any (measurable) subset $I\subset \mathbb{R}^2$. By Kolmogorov theorem, $\mu$ fully determines $F$, so we may simply write $F=F_{\mu}$.  

\subsection{Nazarov-Sodin constant}
Nazarov and Sodin \cite{NS} found the asymptotic law of the expected number of nodal domains of a stationary Gaussian field in growing balls around the origin, provided its spectral measure satisfies certain (simple) properties. We state here a simplified and slightly adapted form of their Theorem, see \cite[Proposition 1.1]{KW2}: 

\begin{thm}
	\label{Nazarov-Sodin}
Let $\mu$ be a probability measure on $\mathbb{S}^1$, invariant by rotation by $\pi$ and let $\mathcal{N}(F_\mu,R)$ be the number of nodal domains of $F_{\mu}$ in a ball of radius $R>0$ centred ad the origin. Then, there exists some constant $c_{NS}(\mu)$ such that 
\begin{align}
	\mathbb{E}[ \mathcal{N}(F_{\mu},R) ]= c_{NS}(\mu)R^2 + O\left(R \right). \nonumber
\end{align}	
Moreover $c_{NS}(\mu)>0$ if $\mu$ does not have any atoms. 
\end{thm}

We will need the following version of Theorem \ref{Nazarov-Sodin}, see  \cite[Proposition 3.4 and Proposition 3.5]{BW}. 
\begin{prop}
	\label{stability}
	Let $R>1$ and $\mu_{f}$ be as in \eqref{spectral measure}. Then, for any function $\psi$ with $||\psi||_{\mathcal{C}^1}$ sufficiently small in terms of $R$, we have 
	\begin{align}
	\mathbb{E}[\mathcal{N}( F_{\mu_{f}} +\psi,R)]= c_{NS}(\mu_{f})R^2 (1+o(1)) \hspace{8mm} \text{as} \hspace{2mm} R\rightarrow \infty.\nonumber
	\end{align}
\end{prop}
We conclude this section mentioning another result concerning the positivity of $c_{NS}(\mu_{f})$. Suppose that $\mu_{f}$ is invariant under $\pi/2$ rotations and reflection on the $X$-axis (i.e. $(x_1,x_2)\rightarrow (x_1,-x_2)$). Among these measures, Kurlberg and Wigman \cite[Theorem 1.5]{KW2} showed that there are only two with vanishing Nazarov-Sodin constant:
\begin{align}
	&\nu= \sum_{k=1}^4 \delta_{e^{i\pi k/2}} 	&\tilde{\nu}= \sum_{k=1}^4 \delta_{e^{i(\pi k/2+\pi/4)}} \nonumber. 
\end{align}

\section{Nodal length of toral eigenfunctions in shrinking sets}
\label{semi-loc}
The aim of this section is to prove the  Proposition \ref{claim} and Proposition \ref{claim2}.  First we show the following consequence of Proposition \ref{claim}: 
\begin{prop}
	\label{semi-locality}
	Let $R>1$, $\epsilon>0$ and let $f$ be as in \eqref{function}. Then, uniformly for $s>E^{-1/2+\epsilon}$ and $z\in \mathbb{T}^2$, we have 
	\begin{align}
	\mathcal{N}_f(s,z)= \frac{E}{R^2}\int_{B(s,z)} \mathcal{N}_f\left(\frac{R}{\sqrt{E}},x\right)dx + O\left(\frac{Es^2}{\sqrt{R}}\right). \nonumber
	\end{align}
\end{prop}
 
\begin{proof}[Proof of Proposition \ref{semi-locality} assuming Proposition \ref{claim}]
	 Let $L>1$ be some parameter to be chosen later. By Proposition \ref{claim} the nodal length of $f$ in $B(s,z)$ is, up to rescaling, at most $\sqrt{E}s$. It follows that there are at most $Es^2/L$ nodal domains of diameter bigger than $L/Es^2$. Therefore,  if we divide $B(s,z)$ into balls of radius $R/\sqrt{E}$, any nodal domain of diameter smaller than $L/Es^2$ intersects at most $O(L^2/R^2)$ balls. We deduce that
	\begin{align}
		\mathcal{N}_f(s,z)= \frac{E}{R^2}\int_{B(s,z)}\mathcal{N}_f\left(\frac{R}{\sqrt{E}},x\right)dx  + O\left(\frac{Es^2}{L}\right) + O\left(\frac{Es^2 L^2}{R^2}\right).  \nonumber
	\end{align} 
The Proposition follows choosing $L= \sqrt{R}$. 	
\end{proof}
\subsection{Proof of Proposition \ref{claim}, upper bound}
The proof will be carried out through a series of lemmas, the first is a standard tool to count zeros of analytic functions. 
\begin{lem}[Jensen's bound]
	\label{Jensen's bound}
	Let $h$ be a complex analytic function on some ball $B \subset \mathbb{C}$ and let $Z(h,\frac{1}{2}\overline{B})$ be the number of its zeros in $\frac{1}{2}\overline{B}$.  Then,
	\begin{align}
	Z\left(h,\frac{1}{2}\overline{B}\right)\ll \log \frac{\underset{B}{\sup} |h|}{\underset{\frac{1}{2}\overline{B}}{\max}|h|}. \nonumber
	\end{align}
\end{lem}
\begin{proof}
	Up to translation and rescaling, we may assume that $h$ is defined on the unit ball, which we again denote by $B$. Let $w_1,...,w_n$ be the zeros of $h$ on $\frac{1}{2}\overline{B}$ counted with multiplicity and consider the Blaschke factor $	D(z,\omega_i)= (z-\omega_i)/(1-z\overline{\omega_i})$. Then, we can write  $h(z)= \prod_{i}D(z,\omega_i) g(z)$ for some $g$ analytic on $B$ with $\underset{B}{\sup} |h|= \underset{B}{\sup} |g|$. Since $|	D(z,\omega_i)| \leq (4/5)$ for $z\in \frac{1}{2}\overline{B}$, letting $Z=	Z\left(h,\frac{1}{2}\overline{B}\right)$, we have 
	\begin{align}
	\underset{\frac{1}{2}\overline{B}}{\max}|h|\leq \left(\frac{4}{5}\right)^{Z}\underset{\frac{1}{2}\overline{B}}{\max}|g|\leq \left(\frac{4}{5}\right)^{Z}\underset{B}{\sup}|g|\leq \left(\frac{4}{5}\right)^{Z}\underset{B}{\sup}|h|.	 \label{9}
	\end{align}
	The lemma follows taking the logarithm on both sides of \eqref{9}. 
\end{proof}
We also need the following well-known formula of Crofton, see for example \cite{F}.
\begin{lem}
	\label{Crofton's formula} Let $f$ be as in \eqref{function}, $s>0$ and $z\in \mathbb{T}^2$, moreover let  $g(y)=f(z+sy)$ for $y\in B(1)$. Then,  uniformly in $s$ and $z$, we have 
	\begin{align}
	\mathcal{L}_f(s,z)s^{-1}=\mathcal{L}(g)\ll \int_{B(1)} \int_{\mathbb{S}^1} Z( g(u+tw)) d\omega du \nonumber
	\end{align}
	where $ Z( g(u+tw))$ is the number of zeros of $g$ as a function of $t\in [0,1]$.
\end{lem}
Finally, we need the following lemma, see \cite{N,T}:
\begin{lem}[Nazarov-Turan]
	\label{NT}
	Let $J\in \mathbb{N}$ and let $h(t)=\sum_{i=1}^{J}a_{\xi}e(\xi_i \cdot t)$ for $t\in \mathbb{C}$ and suppose that $\xi_i \in \mathbb{C}$ are distinct.  Then, for any $B\subset \mathbb{C}$ and $\Omega\subset B$ a measurable subset, we have 
	\begin{align}
	\underset{t\in B}{\sup} |h|<\left(c\frac{|\Omega|}{|B|}\right)^{J-1} e ^{ \max_i |\xi_i| |B|}	\underset{t\in \Omega}{\sup} |h|. \nonumber
	\end{align}
	for some explicit $c>0$. 
\end{lem} 
We are finally ready to prove the upper bound in Proposition \ref{claim}.
\begin{proof}[Proof of the upper bound in Proposition \ref{claim}]
	Let $g(y)=f(z+sy)$ for $y\in B(1)$ and let $h$ be the extension of $g$ to the complex unit ball . By Lemma \ref{Crofton's formula}, we have 
	
		\begin{align}
	\mathcal{L}_f(s,z)s^{-1}=\mathcal{L}(g)\ll \int_{B(1)} \int_{\mathbb{S}^1} Z( g(u+tw)) d\omega du. \label{10.1}
	\end{align}
	 By Lemma \ref{Jensen's bound} and Lemma \ref{NT}, we have 
	\begin{align}
	 Z( g(u+t\omega))\leq Z(h(u+z\omega))\ll	\log \frac{\underset{D}{\sup} |h|}{\underset{\frac{1}{2}\overline{D}}{\max}|h|} \ll N + s\sqrt{E} \leq s\sqrt{E} \label{10}
	\end{align}
uniformly in $u$ and $\omega$. The last inequality in \eqref{10} follows by \eqref{divisor bound} and the fact that $s>E^{-1/2+\epsilon}$. The upper bound then follows by \eqref{10.1} and \eqref{10}. 
\end{proof}

\subsection{Proof of Proposition \ref{claim}, lower bound}
\label{lower bound}
The proof of the lower bound is standard, but we include it for completeness. We need the following result about the density of the zero set: 
\begin{lem}
	\label{density}
Let $f$ be as in \eqref{function}. There exists some absolute constant $c>0$ such that, uniformly for all $z\in \mathbb{T}^2$, the ball $B(c/\sqrt{E},z)$ contains a point where $f$ vanishes. 
\end{lem}
\begin{proof}
Let  $s>0$,	and observe that the function $h(x,t)= f(x)e^{\sqrt{E}t}$ is harmonic in $D=B(s,z) \times [-s,s]$. If $f$ does not vanish in $B(s,z)$, then $h$ is positive; so it satisfies Harnack's inequality:
\begin{align}
	\sup_D |h|\leq C \inf_D |h| \label{16}
\end{align}
for some absolute constant $C>0$. One the other hand, 
\begin{align}
	\sup_D |h| \geq \sup_{B(s,z)}|f| \exp (s\sqrt{E})\geq\inf_{B(s,z)}|f| \exp (s\sqrt{E}) \label{17}
\end{align}
The lemma follows combining \eqref{16} and \eqref{17} and choosing $c$ appropriately.
\end{proof}

We are finally ready to prove the lower bound in Proposition \ref{claim}.
\begin{proof}[Proof of the lower bound  in Proposition  \ref{claim}]
Using Lemma \ref{density}, we can divide $B(s,z)$ in $O(E s^2)$ balls of radius $c/\sqrt{E}$ for some appropriate  $c>0$  such that $f$ vanishes at the centre of each ball. Let $B$ be one of these balls, then the Faber-Krahn inequality \cite[Theorem 1.5]{M} says that every nodal domain has inner radius at least $c_1/E^{1/2}$ for some absolute $c_1>0$, so we have
\begin{align}
	\mathcal{H}^{1}\{ x\in B: f(x)=0\}\gg E^{-1/2}\label{11}
\end{align}
Since \eqref{11} holds for each of the $O(s^2E)$ balls, the lower bound follows. 
\end{proof}

\subsection{Proof of Proposition \ref{claim2}}
As mentioned in the introduction the proof of  Proposition \ref{claim2} follows the proof of Proposition \ref{claim}. We now give some of the details

\begin{proof}[Proof of Proposition  \ref{claim2}]. Let $\mathcal{E}_E= \mathcal{E}:=\{\xi\in \mathbb{Z}^2: |\xi|^2=E\}$, we define an equivalence relation on $\mathcal{E}$ as follows: let $\xi=(\xi^1,\xi^2),\eta=(\eta^1,\eta^2) \in \mathcal{E}$, then $\xi \sim \eta$ if $\xi^1=\pm \eta^2$ and $\xi^2=\pm \eta^2$. Then the general Laplace eigenfunction with eigenvalue $\pi E$ (we make no distinction between $E$ and $\pi E$) satisfying either Dirichlet or Neumann boundary conditions is 
\begin{align}
	&\tilde{f}_{\text{Dirichlet}}(x)= \sum_{ \xi\in \mathcal{E}/ \sim} a_{\xi} \sin( \pi \xi^1 x^1)\sin( \pi \xi^2 x^2) \label{Di} \\
		&\tilde{f}_{\text{Neuman}}(x)= \sum_{ \xi\in \mathcal{E}/ \sim} b_{\xi} \cos( \pi \xi^1 x^1)\cos( \pi \xi^2 x^2) \label{Neu}
\end{align} 
where $x=(x^1,x^2)$. Using the formulas $\sin(a)\sin(b)= 2^{-1}(\cos(a-b)- \cos(a+b))$ and  $\cos(a)\cos(b)= 2^{-1}(\cos(a+b)+ \cos(a-b))$, we can rewrite \eqref{Di} and \eqref{Neu} as 
\begin{align}
		&\tilde{f}_{\text{Dirichlet}}(x)= \sum_{ \xi\in \mathcal{E}/ \sim} \tilde{a}_{\xi} e\left(\langle \xi, x\rangle\right) 
	&\tilde{f}_{\text{Neuman}}(x)= \sum_{ \xi\in \mathcal{E}/ \sim} \tilde{b}_{\xi} e\left(\langle \xi, x\rangle\right)  \nonumber
\end{align}
for some complex coefficients $\tilde{a}, \tilde{b}$. The proof now follows step by step the proof of the upper bound in Proposition \ref{claim}. 
\end{proof}
\section{Bourgain's de-randomisation in shrinking sets}
\label{BourgainS}
Let $R>1$ be fixed, and consider the restriction of $f$, as in \eqref{function}, to a small square centred at $x\in \mathbb{T}^2$:
\begin{align}
	F_x(y)= f\left( x+ \frac{R}{\sqrt{E}}y\right). \label{F}
\end{align}
for $y\in B(1)$. In this section we are going to show that if we sample $x$ uniformly at random from $B(s,z)$, where $z\in \mathbb{T}^2$ and $s>E^{-1/2+\epsilon}$, then the ensemble $\{F_x\}_{x\in B(s,z)}$ approximates the Gaussian field $F_{\mu_{f}}$. The proofs are based on \cite{BU,BW}; nevertheless, the use of Theorem \ref{semi} is required to control the averaging over $B(s,z)$.
\subsection{Approximating $f$ in small squares}
\label{approximating}
In this section, we construct an auxiliary function $\phi_x(y)$ which approximates $F_x(y)$ for most $x\in \mathbb{T}^2$. We begin with some notation: let $K>1$ be some (large) parameter and divide the circle $\mathbb{S}^1$ into arcs $I_k$, of length $1/2K$ for $k\in\{-K,...,K\}$.  Furthermore, let $\delta>0 $ be some (small) parameter and denote by  $\mathcal{K}\subset \{-K,...,K \}$ the subset of indices such that if $k\in \mathcal{K}$ then  
\begin{align}
\mu_f(I_k)>\delta. \label{41}
\end{align}
Finally, let $\mathcal{E}^k=\mathcal{E}_E^k:=\{ |\xi|^2=E: \xi \in I_k\}$ and let $\zeta_k$ be the mid point of $I_k$. 
We are ready to begin the construction, first we re-write $F_x$ as
\begin{align}
F_x(y)= \sum_{k\in \mathcal{K}}\sum_{\xi\in \mathcal{E}^k}a_{\xi}e(\langle \xi, x \rangle)e\left(\left\langle\frac{\xi}{\sqrt{E}},Ry\right\rangle\right)  + \sum_{k\not\in \mathcal{K}}\sum_{\xi\in \mathcal{E}^k}a_{\xi}e(\langle \xi, x \rangle)e\left(\left\langle\frac{\xi}{\sqrt{E}},Ry\right\rangle\right). \label{42}
\end{align}
Second we approximate $\xi/\sqrt{E}$ by $\zeta_k$ for all $\xi\in \mathcal{E}^k$, and define the function
\begin{align}
\phi_x(y)&=\sum_{k\in \mathcal{K}}\left(\sum_{\xi\in \mathcal{E}^k}a_{\xi}e(\langle \xi, x \rangle) \right)e(\langle R\zeta^{k},y \rangle)= \sum_{k\in \mathcal{K}}\mu_{f}(I_k)^{1/2}b_k(x)e(\langle R\zeta^{k},y \rangle) \label{phi}
\end{align}
where
\begin{align}
b_k(x)=\frac{1}{\mu_{f}(I_k)^{1/2}}\sum_{\xi\in \mathcal{E}^{(k)}}a_{\xi}e(\langle\xi,x\rangle). \label{bk}
\end{align}
The following lemma shows that $\phi_x(y)$ is a good approximation to $F_x(y)$ for most $x\in \mathbb{T}^2$.  
\begin{lem}
	\label{first approx}
Let $\epsilon>0$, $R,K,\delta$ be as in Section \ref{approximating}, $F_x$, $\phi_x$ be as in \eqref{F} and \eqref{phi} respectively and $S'$ be defined in Section \ref{NTpre}. Then, for all $E\in S'$ we have 
\begin{align}
	\frac{1}{\pi s^2}\int_{B(s,z)} || F_x- \phi_x||_{\mathcal{C}^1(B(1))} dx \ll R^{6}K\delta  + R^{8}K^{-2} + R^{8}E^{-(1/3)\epsilon} \nonumber
\end{align}
uniformly for all $s>E^{1/2+\epsilon}$ and $z\in \mathbb{T}^2$. 
\end{lem}
\begin{proof}
	Thanks to the Sobolev embedding Theorem, we bound the $\mathcal{C}^1$ norm by the $H^3$ norm 
		\begin{align}
	\frac{1}{\pi s^2}\int_{B(s,z)} || F_x- \phi_x||_{\mathcal{C}^1} dx \ll 	\frac{1}{\pi s^2}\int_{B(s,z)} \int_{B(1)} | D^{\alpha} (F_x (y)- \phi_x(y)|^2 dy \label{49}
	\end{align}
	where $|\alpha|= |(\alpha_1,\alpha_2)| \leq 3$ and $D^{\alpha}= \partial_{\alpha_1}\partial_{\alpha_2}$. First, we estimate the contribution coming from the second term on the right hand side of \eqref{42}. Expanding the square and using the triangular inequality we obtain
	\begin{align}
\frac{1}{\pi s^2}\int_{B(s,z)} \int_{B(1)} \left| D^{\alpha}\sum_{k\not\in \mathcal{K}}\sum_{\xi\in \mathcal{E}^k}a_{\xi}e(\langle \xi, x \rangle)e\left(\left\langle\frac{\xi}{\sqrt{E}},Ry\right\rangle\right)\right|^2 dxdy \nonumber \\
\ll \frac{1}{\pi s^2}R^{2|\alpha|} \sum_{k,k'\not\in \mathcal{K}} \sum_{\substack{\xi\in \mathcal{E}^k \\ \xi'\in \mathcal{E}^{k'}}}| a_{\xi}\overline{a_{\xi'}}|\left| \int_{B(s,z)}e(\langle \xi-\xi', x \rangle)dx \right| .\label{43}  
	\end{align}
Observe that	for $a\in \mathbb{R}^2$ 	
\begin{align}
\int_{B(s,z)} e(\langle a,x\rangle) dx = \begin{cases}
\pi s^2 & a=0 \\
\pi s^2 e(\langle a,z\rangle) \frac{J_1(s|a|)}{s|a|} & a\neq 0 
\end{cases}. \label{orthogonality}
\end{align}
So we separate the terms with $\xi=\xi'$ from the others on the right hand side of \eqref{43} to obtain 
\begin{align}
\eqref{43}\ll	R^{2|\alpha|}\sum_{k\not\in \mathcal{K}}\sum_{\xi\in \mathcal{E}^k} |{a_{\xi}}|^2 + 	R^{2|\alpha|} \sum_{k,k'} \sum_{\xi \neq \xi'} |a_{\xi}a_{\xi'}| \frac{J_1\left( s|\xi-\xi'|\right)}{s|\xi-\xi'|} \label{44}
\end{align}
Since $k\not\in \mathcal{K}$ implies  $\sum_{\xi\in \mathcal{E}^k} |{a_{\xi}}|^2= \mu_{f}(I_k)\leq \delta$, the first term on the right hand side of \eqref{44} is bounded by $R^{2|\alpha|} K \delta$. By Theorem \ref{semi} $|\xi-\xi'|\gg E^{1/2-2\epsilon}$ so  $s|\xi-\xi'|\gg E^{\epsilon}$, it follows that 
\begin{align}  
\frac{J_1\left( s|\xi-\xi'|\right)}{s|\xi-\xi'|} \ll E^{(-2/3)\epsilon} \label{45}
\end{align}
where we have used the bound $J_1(T)\ll T^{-1/2}$ valid for all sufficiently large $T$. Using \eqref{45}, estimating trivially $|a_{\xi}|\leq 1$ and bearing in mind \eqref{divisor bound}, we obtain 
\begin{align}
	\sum_{k,k'} \sum_{\xi \neq \xi'} |a_{\xi}a_{\xi'}| \frac{J_1\left( s|\xi-\xi'|\right)}{s|\xi-\xi'|} \ll E^{(-2/3)\epsilon} \cdot N^2 \ll E^{(-1/3)\epsilon}. \nonumber
\end{align}
All in all, we have shown that 
\begin{align}
\eqref{43} \ll R^{2|\alpha|} K \delta + R^{2|\alpha|}E^{-(1/3)\epsilon}. \label{48}
\end{align}
Now we turn our attention to bounding the difference between $\phi_x$ and the first term on the right hand side of \eqref{42}. Expanding the square and using the triangular inequality, we have
\begin{align}
\frac{1}{\pi s^2}\int_{B(s,z)} \int_{B(1)} \left| D^{\alpha}\sum_{k\in \mathcal{K}}\sum_{\xi\in \mathcal{E}^k} a_{\xi}e( \langle\xi,x\rangle)\left( e\left(\left\langle\frac{\xi}{\sqrt{E}},Ry\right\rangle\right) - e(\langle R\zeta_k, y\rangle) \right) \right|^2 dxdy \nonumber \\
\ll \frac{R^{2|\alpha|+2}}{\pi s^2} \sum_{k,k'} \sum_{\substack{\xi\in \mathcal{E}^k \\ \xi'\in \mathcal{E}^{k'}}}|a_{\xi}a_{\xi'}|\left|\frac{\xi}{\sqrt{E}}-\zeta_k\right|\left|\frac{\xi'}{\sqrt{E}}-\zeta_{k'}\right|\left| \int_{B(s,z)}e(\langle \xi-\xi', x \rangle)dx  \right|. \label{47}
\end{align} 
Similarly to the above, via \eqref{orthogonality} and \eqref{45}, the contribution from the terms with $\xi\neq \xi'$ is at most $R^{2|\alpha|+2} E^{-(1/3)\epsilon} $. The contribution of the terms with $\xi=\xi'$, bearing in mind that $|\xi/\sqrt{E}-\zeta_k|\leq 1/K$, can be bounded by 
\begin{align}
	R^{2|\alpha|+2}\sum_k \sum_{\xi\in \mathcal{E}^k} |a_{\xi}|^2\left|\frac{\xi}{\sqrt{E}}-\zeta_k\right|^2 \ll \frac{	R^{2|\alpha|+2}}{K^2}\sum_k \mu_f(I_k)\leq  \frac{	R^{2|\alpha|+2}}{K^2}.\label{46} 
\end{align}
All in all we have, 
\begin{align}
	\eqref{47} \ll \frac{	R^{2|\alpha|+2}}{K^2} + R^{2|\alpha|+2} E^{-(1/3)\epsilon}. \label{50}
\end{align}
The lemma follows combining \eqref{49}, \eqref{48} and \eqref{50}. 
\end{proof}
\subsection{Gaussian moments}
Recall the notation \eqref{bk}, we are going to show that the vector $(b_k)_{k\in \mathcal{K}}$ approximates a Gaussian vector $(c_k)_{k\in \mathcal{K}}$, where $c_k$ are i.i.d. complex standard Gaussian random variables subject to $\overline{c}_k=c_{-k}$. We prove the following quantitative lemma:  
\begin{lem}
	\label{independence in shrinking}
	Let  $\epsilon>0$, $b_k$ be as in \eqref{bk} and $\mathcal{K},K,\delta$ be as in Section \ref{approximating}. Moreover let $B$ be some large parameter and fix two sets of positive integers $\{r_k\}_{k\in \mathcal{K}}$ and $\{s_k\}_{k\in \mathcal{K}}$ such that $\sum r_k +  s_k \leq B$. Suppose that $E \in S'$, then 
	\begin{align}
	\left|\frac{1}{\pi s^2}\int_{B(s,z)}\prod_{k\in \mathcal{K}}b_k^{r_k}\overline{b}_k^{s_k}dx- \mathbb{E}\left[\prod_{k\in \mathcal{K}}c_k^{r_k}\overline{c}_k^{s_k}\right]\right|= o_{\delta,K,B}(1) \hspace{5mm} \text{as} \hspace{2mm} N\rightarrow \infty \nonumber
	\end{align}
	uniformly for $f$  flat, $s>E^{-1/2+\epsilon}$ and $z\in \mathbb{T}^2$. 
\end{lem}
\begin{proof}
Expanding the product, we have 
	\begin{align}
			\frac{1}{\pi s^2}	\int_{B(s,z)}\prod_{k}b_k^{r_k}\overline{b}_k^{s_k}dx= \prod_{k} \mu_{f}(I_k)^{-(r_k+s_k)/2}\sum \left( \prod_{i=1}^{r_k}\prod_{j=1}^{s_k}a_{\xi_{i,k}}\overline{a}_{\xi'_{j,k}}\right) \times \nonumber\\  \times \int_{B(s,z)} e\left( \Bigl\langle \sum_{k,i,j}  (\xi_{i,k}- \xi'_{j,k}),x \Bigr\rangle\right)dx \label{formula1}
	\end{align}
where the out most sum is over all the choices $\xi_{1,1},...,\xi_{1,r_1}, \xi'_{1,1}...,\xi'_{1,s_1},..., \xi_{k,1},..,.\xi_{k,r_k}, \xi'_{k,1}, \\ ...,\xi'_{1,s_k}$. We split the sum in \eqref{formula1} according to \eqref{orthogonality}: we first consider the contribution from the \emph{constant term} $\sum_{k,i,j}(\xi_{i,k}- \xi'_{j,k})=0$ and then the contribution from the \emph{oscillatory term} $|\sum_{k,i,j}(\xi_{i,k}- \xi'_{j,k})|>0$. Furthermore, we subdivide the constant term into \textquotedblleft diagonal" solutions, namely $\{\xi_{i,k}\}= \{\xi'_{j,k}\}$ for each $k\in \mathcal{K}$, and all the other solutions, which we call \textquotedblleft off-diagonal".
	
	\textbf{Constant term}, \textit{\textquotedblleft diagonal" solutions}.
	If  $\{\xi_{i,k}\}= \{\xi'_{j,k}\}$, then $r_k= s_k$, so, taking into account the possible rearrangements and by definition of $I_k$ and $\mu_{f}(I_k)$, we have a contribution to the  right hand side of \eqref{formula1} of 
	\begin{align}
	g_k:=r_k! \cdot \mu_{f}(I_k)^{-r_k}\sum_{ \{\xi_{i,k}\}= \{\xi'_{j,k}\} } \prod_{i=1}^{r_k} |a_{\xi_{i,k}}|^2	=  \mathbb{E}[ |c_k|^{2r_k}]. \nonumber
	\end{align}
    Multiplying together the contributions from all $k$'s, we obtain 
    \begin{align}
    	\prod_{k} g_k= \mathbb{E}\left[\prod_{k}|c_k|^{2 r_k}\right].\label{3.1}
    \end{align}
    
    	\textbf{Constant term}, \textit{\textquotedblleft off-diagonal" solutions}.
    	Let $B_1=\sum_{k}r_k+ s_k$. Since $E\in S'$, Theorem \ref{BB} implies that the number of off-diagonal solutions is at most $O(N^{\gamma B_1})$. Since $\mu_{f}(I_k)\geq \delta$ for all $k \in \mathcal{K}$ and $|a_{\xi}|^2\leq N^{-1+o(1)}$, we obtain as $N\rightarrow \infty$
    	\begin{align}
    \left|\sum_{\text{off-diagonal}}\prod_{k} \mu_{f}(I_k)^{r_k+s_k/2}\sum \prod_{k} \prod_{i=1}^{r_k}\prod_{j=1}^{s_k}a_{\xi_{i,k}}\overline{a}_{\xi_{j,k}}\right|\ll_{K} N^{-B_1/2+ \gamma B_1 +o(1)} \delta^{-B_1/2}=o_{\delta, B, K}(1).  \label{3.2}
    	\end{align}
    	
    \textbf{Oscillatory term}. If 	$|\sum_{k,i,j}(\xi_{i,k}- \xi'_{j,k})|>0$, Theorem \ref{semi} implies that $|\sum_{k,i,j}(\xi_{i,k}- \xi'_{j,k})|>E^{1/2-2\epsilon}$, therefore $s|\sum_{k,i,j}(\xi_{i,k}- \xi_{j,k})|>E^{\epsilon}$. So, bearing in mind that  $J_1(T) \ll T^{1/2}$ for $T$ sufficiently large, we have 
    \begin{align}
    	\frac{J_1(s|\sum_{k,i,j}(\xi_{i,k}- \xi_{j,k})|)}{s|\sum_{k,i,j}(\xi_{i,k}- \xi_{j,k})|}\ll E^{-2/3 \epsilon} .\label{3.5}
    \end{align}
    	Since the maximum number of terms in the outer sum in \eqref{formula1} is $N^{B_1}$, $\mu_{f}(I_k)\geq \delta$ for all $k \in \mathcal{K}$, $|a_{\xi}|\leq N^{-1+o(1)}$ and bearing in mind \eqref{3.5} and \eqref{divisor bound} , we obtain as $N\rightarrow \infty$
    	
    	\begin{align}
    		\prod_{k} \mu_{f}(I_k)^{-(r_k+s_k)/2}\sum \left( \prod_{k} \prod_{i=1}^{r_k}\prod_{j=1}^{s_k}a_{\xi_{i,k}}\overline{a}_{\xi_{j,k}}\right)		\frac{J_1(s|\sum_{k,i,j}(\xi_{i,k}- \xi_{j,k})|)}{s|\sum_{k,i,j}(\xi_{i,k}- \xi_{j,k})|} \nonumber \\ \ll_{K} N^{B_1/2 +o(1)} \delta^{-B_1/2} E^{-2/3 \epsilon}=o_{\delta,K,B_1}(1). \label{3.3}
    	\end{align} 
 The Proposition follows combining \eqref{3.1}, \eqref{3.2}, \eqref{3.3}. 
\end{proof}
 Lemma \ref{independence in shrinking}, by the method of moments, implies that the vector $(b_k)_{k\in \mathcal{K}}$ converges in distribution to the vector $(c_k)_{k\in \mathcal{K}}$ . We restate this fact in the following convenient way, more details can be found in \cite[Lemma 6.5]{BW} and \cite[Page 9]{BU}, see in particular \cite[Lemma 6.4]{BW} for the fact that the measure induced by the $b_k$'s is absolutely continuous with respect to the Lebesgue measure.  

\begin{cor}
	\label{tau}
Let $\epsilon>0$ and $\alpha_1,\alpha_2>0$ be given, let $\delta,K,B,$ as in Lemma \ref{independence in shrinking} and $f$ be as in \eqref{function}. Suppose that $E \in S'$ is sufficiently large depending on $\epsilon,\alpha_1,\alpha_2,K,\delta$ and $B$. Then, uniformly for all $f$  flat, $s>E^{-1/2+\epsilon}$ and $z\in \mathbb{T}^2$, there exists a measurable map $\tau: \Omega\rightarrow B(s,z)$  and a subset $\Omega^1\subset \Omega$ with the following properties: 

\begin{enumerate}
	\item For any measurable $A\subset \Omega^1$, we have $\vol(\tau(A))= \pi s^2\mathbb{P}(A)$.
	\item  $\mathbb{P}(\Omega^1)>1-\alpha_1$ .
	\item For all $\omega\in \Omega^1$, we have $|b_k(\tau(\omega))- c_k(\omega)|\leq \alpha_2$ uniformly for all $k\in \mathcal{K}$. 
\end{enumerate}
\end{cor}
\subsection{Discarding $\phi_x$}
Before proving the main result of this section, we need the following lemma: 
\begin{lem}[Lemma 4, \cite{SO}]
	\label{Sodin}
	Let $R>1$ $\alpha_3,\alpha_4>0$, $\{\mu_n\}_{n\in \mathbb{N}}$ be a sequence of probability measures on $\mathbb{S}^1$ such that $\mu_n$ weak$^{\star}$ converges to some probability measure $\mu$. Then,  for all $n$ sufficiently large depending on $\alpha_3,\alpha_4$ and $R$, we have  
	\begin{align}
||F_{\mu_n}- F_{\mu}||_{\mathcal{C}^1(B(R))}\leq \alpha_3 \nonumber
	\end{align}
	outside an event of probability $\alpha_4$. 
\end{lem}

\begin{proof}
	We can associate to $\mu$ the Gaussian measure $G$ defined on $\mathbb{R}^2$ as follows: for any open and measurable (with respect to $\mu$) subset $A$ of $\mathbb{R}^2$ we let
	\begin{align}
	G(A)= N(0,\mu(A))\nonumber
	\end{align}
	where $N(0,\mu(A))$ is a Gaussian random variable with mean zero and variance $\mu(A)$. Moreover, if $A\cap B= \emptyset$, we require $G(A)$ and $G(B)$ to be independent. We define $G_n$ with respect to $\mu_n$ similarly. Since $\mu$ is compactly supported, we see that $G_n$ weak$^{\star}$ converges to $G$ and, since a normal random variable is square integrable, we obtain $G_n\rightarrow G$ in $L^2(\Omega)$ (recall that $\Omega$ is the common probability space of our random objects). By \cite[Theorem 5.4.2]{AT}, we have the $L^2(\Omega)$ representations 
	\begin{align}
	&F_{\mu_n}(x)= \int_{\mathbb{S}^1}e(\langle x\cdot\lambda\rangle)G_n(d\lambda) &F_{\mu}(x)= \int_{\mathbb{S}^1}e(\langle x\cdot\lambda\rangle)G(d\lambda). \label{13}
	\end{align}
	Since $\mu$ and $\mu_n$ are compactly supported, we can differentiate under the integral in \eqref{13}; bearing in mind that $G_n$ weak$^{\star}$ converges to $G$, we have 	$||F_{\mu_n}- F_{\mu}||_{\mathcal{C}^1(B(R))}\rightarrow 0$ as $n\rightarrow \infty$ in $L^2(\Omega)$. This implies the conclusion of the Lemma. 
\end{proof}
We are finally ready to state and prove the main result of this section: 
\begin{prop}
	\label{main prop}
Let $\epsilon>0$, $R>1$ and $\eta_1,\eta_2>0$, $f$ be as in \eqref{function}.  Suppose that $E \in S'$ is sufficiently large depending on $\epsilon,\eta_1,\eta_2$ and $R$. Then, uniformly for all $f$  flat, $s>E^{-1/2+\epsilon}$ and $z\in \mathbb{T}^2$, there exists a measurable map $\tau: \Omega\rightarrow B(s,z)$  and a subset $\Omega'\subset \Omega$ with the following properties: 

\begin{enumerate}
	\item For any measurable $A\subset \Omega$, we have $\vol(\tau(A))= \pi s^2\mathbb{P}(A)$.
	\item  $\mathbb{P}(\Omega')>1-\eta_1$ .
	\item For all $\omega\in \Omega'$, we have	$||F_{\tau(\omega)}(y)- F_{\mu_{f}}(Ry,w)||_{ \mathcal{C}^1(B(1))}\leq \eta_2$ 
\end{enumerate}

\end{prop}
\begin{proof}
Let $\mathcal{K},K,\delta$ be as in Section \ref{approximating} and let $F_K(Ry,\omega):=\sum_{k\in \mathcal{K}}\mu_{f}(I_k)^{1/2} c_k(\omega) e(\langle \zeta_k,Ry\rangle)$. Thanks to Corollary \ref{tau} with $\alpha_1=\eta_1/3$ and $\alpha_2=1/K^2$, there exist  $\tau: \Omega\rightarrow B(s,z)$ and $\Omega^1\subset \Omega$ such that: 
\begin{itemize}
	\item For any measurable $A\subset \Omega'$, we have $\vol(\tau(A))= \pi s^2\mathbb{P}(A)$.
	\item  $\mathbb{P}(\Omega^1)>1-\eta_1/3$  .
	\item For all $\omega\in \Omega'$, we have	
	\begin{align}
	||\phi_{\tau(\omega)}(y) -F_K(Ry,\omega)||_{\mathcal{C}^1}\ll RK \alpha_2= \frac{R}{K}\leq \eta_2/3 \label{main1}
	\end{align}
 provided $K$ is sufficiently large depending on $R$ and $\eta_2$. 
\end{itemize}

\begin{claim}
	\label{claim1}
	 There exists some $\Omega^2\subset \Omega$ with $\mathbb{P}(\Omega^2)>1-\eta_1/3$ such that 
	\begin{align}
	||F_K(Ry,\omega)- F_{\mu_{f}}(Ry,\omega)||_{\mathcal{C}^1}\leq \eta_2/3\label{main2}
	\end{align}
	for all $K$ sufficiently large depending on $\eta_1,\eta_2$ and $R$.
\end{claim}  
	 To prove the claim, observe that $F_K$ is a Gaussian field with spectral measure 
	\begin{align}
	\mu_K= \sum_{k\in \mathcal{K}} \mu_{f}(I_k) \delta_{\zeta_k}. \nonumber
	\end{align}
	By definition of $\mu_f$, we have $\sup_{A\subset \mathbb{S}^1}	|\mu_{f}(A)- \mu_K(A)| \ll \delta K $. So, taking $\delta<1/K^2$ and $K$ sufficiently large, the claim follows from Lemma \ref{Sodin}.
 
 \vspace{2mm}
  
  Finally, by Lemma \ref{first approx} and Markov's inequality, we have 
  \begin{align}
   || F_x- \phi_x||_{\mathcal{C}^1}  \leq \eta_2/3 \label{main3}
   \end{align}
  for all $x\in B\subset B(s,z)$, where
  \begin{align}
  	  (\pi s^2)^{-1} \vol(B(s,z) \backslash B)\ll \eta_2^{-1}\left( R^{6}K\delta  + R^{8}K^{-2} + R^{8}E^{-(1/3)\epsilon} \right) \leq \eta_1/3 . \label{5.1}
  \end{align}
  for $K$ and $E$ sufficiently large in terms of $R$, $\eta_1$ and $\eta_2$. We briefly summaries our choices of parameters: $R,\eta_1,\eta_2$ are fixed, $\delta<1/K^2$, $K$ is large depending $R,\eta_1,\eta_2$ and $E$ is large depending on $R,\eta_1,\eta_2$ and $K$. We are now ready to conclude the proof. Let $\Omega'= \Omega^1 \cap \Omega^2 \cap \tau^{-1} (B)$, then $\tau$ restricted to $\Omega'$ satisfies $(1)$. By Corollary \ref{tau}, Claim \ref{claim1} and \eqref{5.1}, we also have $\mathbb{P}(\Omega')\geq 1- \eta_1$ so $(2)$ holds. Finally, $(3)$ follows by \eqref{main1}, \eqref{main2} and \eqref{main3}, valid for all $\omega \in \Omega'$. This concludes the proof of the Proposition. 
\end{proof}

\section{Concluding the proof}
\label{end}
We are finally ready to prove Theorem \ref{theorem 3}. 
\begin{proof}[Proof of Theorem \ref{theorem 3}]
 Pick some $\eta_1, \eta_2>0$  to be chosen later, and let $\tau$ and $\Omega'$ be given by Proposition \ref {main prop}.  Then, by Proposition \ref{semi-locality}, we have 
	\begin{align}
	\mathcal{N}_f(s,z)&= \frac{E}{R^2}\int_{B(s,z)} \mathcal{N}( F_x)dx + O\left(\frac{Es^2}{\sqrt{R}}\right)   \nonumber \\
	&= \frac{E}{R^2}\int_{\tau^{-1}(\Omega')} \mathcal{N}( F_x)dx + \frac{E}{R^2} \int_{B(s,z)\backslash \tau^{-1}(\Omega')} \mathcal{N}(F_x)dx + O\left(\frac{Es^2}{\sqrt{R}} \right). \label{53} 
	\end{align}
By part $(3)$ of Proposition \ref{main prop}, we may write $F_x(y)=   	F_{\mu_{f}}(Ry) +\psi$ for $x\in \tau^{-1}(\Omega')$ and some function $\psi$ with $||\psi||_{\mathcal{C}^1}\leq \eta_1$. Thus, we can rewrite \eqref{53}, bearing in mind part $(1)$ of Proposition \ref{main prop}, as
	\begin{align}
	\mathcal{N}_f(s,z)&= \frac{\pi s^2 E}{R^2}\int_{\Omega'} \mathcal{N}( F_{\mu_{f}} +\psi, R )d\omega  + \frac{E}{R^2}\int_{B(s,z)\backslash \tau^{-1}(\Omega')} \mathcal{N}(F_x)dx  + O\left(\frac{Es^2}{\sqrt{R}}\right) \nonumber \\
	&= \frac{ \pi s^2 E}{R^2}\left(\int_{\Omega} \mathcal{N}( F_{\mu_{f}} +\psi, R ) d\omega- \int_{\Omega\backslash \Omega'} \mathcal{N}( F_{\mu_{f}} +\psi, R ) d\omega \right)  \nonumber \\
	 & \hspace{1cm} + \frac{  E}{R^2}\int_{B(s,z)\backslash \tau^{-1}(\Omega')} \mathcal{N}(F_x)dx + O\left(\frac{Es^2}{\sqrt{R}}\right) \label{52}
	\end{align}
	where in the second equality we set $\psi(y,\omega)=0$ for $\omega \not \in \Omega'$. By the Faber-Krahn inequality, $\mathcal{N}(F_x)\ll R^2$, and since $F_{\mu_{f}} +\psi= F_{\tau(\omega)}$, we have $ \mathcal{N}( F_{\mu_{f}} +\psi ) \ll R^2 $ uniformly for all $\omega \in \Omega'$. For $\omega \not \in \Omega'$, by definition, we have $\Delta F_{\mu_{f}}=-R^2F_{\mu_{f}}$. Thus, again by the Faber-Krahn inequality, $ \mathcal{N}( F_{\mu_{f}} +\psi ) \ll R^2 $ holds uniformly for all $\omega \in \Omega$.  Therefore, taking $\eta_2= 1/\sqrt{R}$, the second and third integrals in \eqref{52} are bounded by $O(R^{3/2})$. Thus, \eqref{52} can be re-written as 
	\begin{align}
			\mathcal{N}_f(s,z)=\frac{\pi s^2 E}{R^2}\mathbb{E}[ \mathcal{N}( F_{\mu_{f}} +\psi, R )] +O\left(\frac{Es^2}{\sqrt{R}}\right). \nonumber
	\end{align}
	Taking $\eta_1$ small enough in terms of $R$ via Proposition \ref{stability} and then taking $R\rightarrow \infty$, we deduce  
		\begin{align}
	\mathcal{N}_f(s,z)= c_{NS}(\mu_{f})\pi s^2 E(1+ o_{R\rightarrow \infty}(1))  . \nonumber
	\end{align}
	and the Theorem follows by taking $R$ to be an arbitrarily slowly, depending on all the parameters, growing function of $E$. 
\end{proof}

\section*{Acknowledgement}
The author would like to thank Igor Wigman for many useful discussions. Alejandro Rivera and Maxime Ingremeau for stimulating conversations and Priya Lakshmi for her comments on an early draft of this article. This work was supported by the Engineering and Physical Sciences Research Council [EP/L015234/1].  
The EPSRC Centre for Doctoral Training in Geometry and Number Theory (The London School of Geometry and Number Theory), University College London.


\begin{thebibliography}{9}
					\bibitem{AT}
		R. J. Adler, J. E. Taylor,  \emph{ Random fields and geometry}, Springer Monographs in Mathematics. Springer, New York, (2007).
		
		\bibitem{BMW}
		J. Benatar, D. Marinucci and I. Wigman  \emph{ Planck-scale distribution
			of nodal length of arithmetic random waves},  Preprint arXiv:1710.06153 (2017).
		
		\bibitem{B1}
		M. Berry, \emph{Regular and irregular semiclassical wavefunctions}, J. Phys. A 10, no. 12, (1977).
		
		\bibitem{B2}
		M. Berry, \emph{ Semiclassical mechanics of regular and irregular motion. Chaotic behavior of deterministic systems} Les	Houches, 171-271,  (1981).
		
		\bibitem{BB}
		E. Bombieri, J. Bourgain, \emph{A problem on sums of two squares}, IMRN 11, 3343-3407 (2015).
		
		
		\bibitem{BU}
		J. Bourgain, \textit{On toral eigenfunctions and the random wave model}. Israel J. Math. 201, no. 2, 611-630, (2014). 
		
		\bibitem{BW}
		J. Buckley,  I. Wigman, \emph{On The Number Of Nodal Domains Of Toral Eigenfunction}.
		Ann. Henri Poincar\'{e} 17.11 , 3027-3062.(2016).
		
		\bibitem{CKW}
		V. Cammarota, O. Klurman, I. Wigman, \emph{Boundary effect on the nodal length for Arithmetic Random Waves, and spectral semi-correlations}, arxiv-preprint https://arxiv.org/abs/1903.10602
		
		\bibitem{C}
		S. Y.	Cheng, \emph{ Eigenfunctions and nodal sets}. Comm. Math. Helv. 51 , 43-55 (1976). 
		
		\bibitem{CI}
		J. Cilleruelo, \emph{ The distribution of the lattice points on circles}, J. Number Theory 43, no. 2, 198-202, (1993). 
		
		
		
		\bibitem{DV}
		Y. Colin de Verdi\`{e}re, \emph{Ergodicit\`{e} et fonctions propres du Laplacien}, Comm. Math. Phys. 102, 497-502, (1985)
		
		\bibitem{CH}
		R. Courant,   D. Hilbert, \emph{ Methoden der mathematischen Physik. I.} (German) Dritte Auflage. Heidelberger Taschenbücher, Band 30. Springer-Verlag, (1968).
		
		\bibitem{DF}
		H. Donnelly, C. Fefferman \emph{ Nodal sets of eigenfunctions on Riemannian manifolds.}
		Invent. Math. 93 , 161-183. (1988). 
		
		\bibitem{E}
		P. D. T. A. Elliott, \emph{ Probabilistic number theory. II. Central limit theorems.} , Fundamental Principles of Mathematical Sciences, Springer-Verlag, (1980).
		
		\bibitem{F}
		H. Federer, \emph{ Geometric measure theory}, Berlin-Heidelberg-New York, Springer, (1969). 
		
		\bibitem{HR}
		H. Hezari, G. Rivi\`{e}re, \emph{$L^p$
			norms, nodal sets, and quantum ergodicity}, Adv. Math. 290, 938-966,  (2016).
		
		\bibitem{H}
		X. Han, \emph{Small scale quantum ergodicity in negatively curved manifolds,} Nonlinearity 28, no. 9, (2015).
		
		\bibitem{I}
		M. Ingremeau \emph{Lower bounds for the number of nodal domains for sums of two distorted plane waves in non-positive curvature}, https://arxiv.org/abs/1612.01911. 
		
		\bibitem{JZ}
		J. Jung, S. Zelditch, \emph{ Number of nodal domains and singular points of eigenfunctions of negatively curved surfaces with an isometric involution}, J. Differential Geom. 102, no. 1, 37-66, (2016).
		
		\bibitem{GW}
		A. Granville, I. Wigman, \emph{Planck-scale mass equidistribution of toral Laplace eigenfunctions}, Communications in Mathematical
		Physics, 355(2), (2017) .
		
		\bibitem{KW}
		P. Kurlberg , I. Wigman, \emph{ On Probability Measures Arising From Lattice Points On Circles}. Mathematische Annalen (2016).
		
		\bibitem{KW2}
		P. Kurlberg, I. Wigman, \emph{ Variation of the Nazarov-Sodin constant for random plane waves and arithmetic random waves}. Adv. Math. , 516-552, (2018).
		
		\bibitem{LE}
		M. Ledoux, \emph{ The concentration of measure phenomenon}, American Mathematical Society, Providence, RI, (2001).
		
		\bibitem{LR}
		S. Lester, Z. Rudnick, \emph{Small scale equidistribution of eigenfunctions on the torus}. Comm. Math. Phys. 350, no. 1, 279-300, (2017). 
		
		\bibitem{LEW}
		H. Lewy, \emph{On the minimum number of domains in which the nodal lines of spherical
			harmonics divide the sphere}, Comm. Partial Differential Equations 2, no. 12,
		1233-1244, (1977). 
		
		\bibitem{LM}
		A. Logunov, E. Malinnikova, \emph{Nodal sets of Laplace eigenfunctions: estimates of the Hausdorff measure in dimensions two and three}, 50 years with Hardy spaces, 333-344, (2018).
		
		\bibitem{L1}
		A. Logunov, \emph{Nodal sets of Laplace eigenfunctions: proof of Nadirashvili's conjecture and of the lower bound in Yau's conjecture}, Ann. of Math. (2) 187 , no. 1, 241-262, (2018). 
		
		\bibitem{L2}
		A. Logunov, \emph{ Nodal sets of Laplace eigenfunctions: polynomial upper estimates of the Hausdorff measure}, Ann. of Math. (2) 187, no. 1, 221-239, (2018). 
		
		\bibitem{LS}
		W. Z. Luo, P. Sarnak, \emph{ Quantum ergodicity of eigenfunctions on $PSL_2(Z)\backslash H^2$
		}, Inst. Hautes Etudes Sci. Publ.  Math, no. 81, 207-237,  (1995). 
		
		\bibitem{M}
		D. Mangoubi, \emph{ Local asymmetry and the inner radius of nodal domains} . Comm. Partial Differential Equations 33, 7-9, 1611-1621, (2008).
		
		
		
		\bibitem{N}
		F. Nazarov, \emph {Local estimates for exponential polynomials and their applications to
			inequalities of the uncertainty principle type.} Algebra i Analiz 5, no. 4, 3–66, (1993).
		
		\bibitem{NS}
		F. Nazarov, M. Sodin, \emph{ Asymptotic laws for the spatial distribution and the number of connected components of zero sets of Gaussian random functions}. Zh. Mat. Fiz. Anal. Geom. 12, 205-278, (2016). 
		
		\bibitem{HU}
		P.  Humphries,  \emph{Equidistribution in Shrinking Sets and L
			4-Norm Bounds for Automorphic Forms}, Math. Ann. 371, no. 3-4, 1497-1543, (2018).
		
		\bibitem{HU1}	
		P. Humphries, R. Khan, \emph{On the Random Wave Conjecture for Dihedral Maaß Forms}, to appear in Geometric and Functional Analysis (2019). 
		
		\bibitem{SA1}
		A. Sartori \emph{Mass distribution for toral eigenfunctions via Bourgain's de-randomisation},  Q. J. Math. (to appear), https://arxiv.org/abs/1812.00962. 
		
		\bibitem{SA}
		A. Sartori \emph{Quasi-spectral correlations and phase-transition for arithmetic random waves} in preparation
		
		\bibitem{SA2}
		A. Sartori, \emph{ On the fractal structure of attainable probability measures}, Bull. Pol. Acad. Sci. Math. 66, no. 2, 123-133, (2018). 
		
		\bibitem{S}
		A. Snirel’man, \emph{Ergodic properties of eigenfunctions}, Uspekhi Mat. Nauk 180, 181-182, (1974).
		
		\bibitem{SO}
		M. Sodin, \emph{Lectures on random nodal portraits}, Probability and statistical physics in St. Petersburg, 91, 395-422 (2016).
		
		
		\bibitem{ST}
		A. Stern, \emph{ Bemerkungen uber asymptotisches Verhalten von Eigenwerten und Eigenf\"{u} nctionen} (German) InauguralDissertation zur Erlangung der Doktorw\"{u}rde der Hohen Mathematisch-Naturwissenschaftlichen Fakult\"{a}t der
		Georg August-Universitat zu Gottingen, (1924 ).
		
		\bibitem{T}
		P. Tur\'{a}n, \emph{On the distribution of zeros of general exponential polynomials},
		Publ. Math. Debrecen 7, 130-136, (1960). 
		
		\bibitem{WY}
		I. Wigman, N. Yesha  \emph{CLT for mass distribution of Toral Laplacian eigenfunctions}, Mathematika 65, no. 3, 643–676, (2019).
		
		
		\bibitem{Y}
		M. Young,  \emph{The quantum unique ergodicity conjecture for thin sets}, Adv. Math. 286, 958-1016, (2016).
		
		\bibitem{Z}
		S. Zelditch, \emph{Uniform distribution of eigenfunctions on compact hyperbolic surfaces}, Duke Math. J. 55, 919-941, (1987). 
	

	\end{thebibliography}
\end{document}